\documentclass[11pt]{amsart}
\usepackage [T1]{fontenc}
\usepackage [latin1]{inputenc}
\usepackage{amssymb,amsmath,amsthm}
\usepackage{verbatim}
\usepackage{graphicx}
\usepackage{epsfig}

\usepackage{color, soul}
\usepackage[all]{xy}

\numberwithin{equation}{section}

\definecolor{dgreen}{cmyk}{1,.5,1,.2}

\newtheorem{theorem}{Theorem}[section]

\newtheorem{lemma}[theorem]{Lemma}
\newtheorem{proposition}[theorem]{Proposition}
\newtheorem{corollary}[theorem]{Corollary}

\def\s{\backslash}
\def\os{\otimes^{n,s}}
\def\osc{\widetilde{\otimes}^{n,s}}

\def\Q{\mathcal{Q}}
\def\q{\mathcal{Q}}
\def\K{\kappa}
\def\P{\mathcal{P}^n}
\def\p{\mathcal{P}^n}

\def\tsty{\textstyle}

\begin{document}

\title[The Five Basic Lemmas]{Five basic lemmas for symmetric tensor products of normed spaces}

\author{Daniel Carando}
\thanks{Partially supported by UBACyT Grant X218 and CONICET PIP 0624.}

\author{Daniel Galicer}

\thanks{Partially supported by UBACyT Grant X038, CONICET PIP 0624 and a doctoral fellowship from CONICET.}

\address{Departamento de Matem\'{a}tica, Facultad de Ciencias Exactas y Naturales, Universidad de Buenos Aires,\\ Pab. I, Cdad Universitaria (1428)  Buenos Aires, Argentina and CONICET.} \email{dcarando@dm.uba.ar}
\email{dgalicer@dm.uba.ar}

\begin{abstract} We give the symmetric version of five lemmas which are essential for the theory of tensor products (and norms). These are: the approximation, extension, embedding, density and local technique lemmas. Some application of these tools to the metric theory of symmetric tensor products and to the theory of polynomials ideals are given.
\end{abstract}



\maketitle

\section*{Introduction}

Grothendieck, in its \emph{``R\'{e}sum\'{e} de la th\'{e}orie m\'{e}trique des produits tensoriels
topologiques''} \cite{Groth53}, created the basis of what was later known as `local theory', and exhibited the importance of the use of tensor products in the theory of Banach spaces and Operator ideals. Tensor products had appeared in functional analysis since the late thirties, in works of Murray, Von Neumann and Schatten (see \cite{Sch50}). But it was Grothendieck who realized the local nature of many properties of tensor products, and this allowed him to establish a very useful theory of duality. In 1968, Grothendieck's resum\'{e} was fully appreciated when Lindenstrauss and Pe{\l}czy\'nski \cite{LinPel68}
presented important applications to the theory of absolutely $p$-summing operators, translating ideas written in terms of tensor norms by Grothendieck, into properties of operator ideals. By the same time, a general theory of operator ideals on the class of Banach spaces was developed by  Pietsch and his school, without the use of tensor norms \cite{Pie78}. As stated by Defant and Floret in their book \cite{DefFlo93}, ``both theories, the theory of tensor norms and of norm operator ideals, are more easily understood and also richer if one works with both simultaneously''.

In 1980, Ryan introduced symmetric tensor products of Banach spaces, as a tool  for the study of polynomials and holomorphic mappings \cite{Rya80}. Since then, many steps were given towards a metric theory of symmetric tensor products and a theory of polynomial ideals. As in the linear case, both theories influence and contribute to each other. In his survey \cite{Flo97}, Floret presented the algebraic basics of symmetric tensor products, together with a thorough account of fundamental metric results for the two extreme tensor norms: the symmetric projective tensor norm $\pi_s$ and the symmetric injective tensor norm $\varepsilon_s$. Unfortunately, there is not such a treatise on general symmetric tensor norms, although  the theory of symmetric tensor products and polynomial ideals steadily evolved in the last decades, as we can see in the (rather incomplete) list \cite{BotBraJunPel06,BoyLass05,BoyLas10,BoyRya01,CarDimMur09,Car-GalPrims,CarGal11,DefDiaGarMae01,Flo01,Flo01b,Flo02(On-ideals),Flo02(Ultrastability),FloGar03}. This note aims to contribute to a systematic development of the theory of symmetric tensor products of Banach spaces and polynomial ideals.

In the theory of full 2-fold tensor norms, ``The Five Basic Lemmas'' (see Section 13 in Defant and Floret's book \cite{DefFlo93}) are rather simple results which turn out to be ``basic for the understanding and use of tensor norms''.
Namely, they are the Approximation Lemma, the Extension Lemma, the Embedding Lemma, the Density Lemma and the Local Technique Lemma. Applications of these lemmas can be seen throughout the book.
We present here the analogous results for the symmetric setting. We will also exhibit some applications as example of their potential. In order to obtain our five basic lemmas and their applications, we combine new and known results in a methodic way, following the lines of \cite{DefFlo93}. Although some proofs are similar to the 2-fold case, the symmetric nature of our tensor products introduces some difficulties, as we can see, for example, in the symmetric version of the Extension Lemma \ref{extension lemma}, whose proof is much more complicated than that of its full 2-fold version.

The article is organized as follows. In Section 1 we give the definitions and some general results that will be used in the sequel. In Section 2 we state and prove the five basic lemmas, together with some direct consequences. Applications are given in Sections 3 and 4. Section 3 collects those applications on the metric theory of symmetric tensor norms. In Section 4 we consider applications to Banach polynomial ideals. In particular, we reformulate some results of the previous sections in terms polynomial ideals.

\smallskip
We refer to \cite{DefFlo93, Rya02} for the theory of tensor norms and operator ideals, and to \cite{Flo97,Flo01,Flo02(On-ideals),Flo02(Ultrastability)} for symmetric  tensor products and polynomial ideals.

\section{Preliminaries}

Throughout the paper $E$ and $F$ will be real or complex normed spaces, $E'$ will denote the dual space of $E$, $\K_E: E \longrightarrow E''$ the canonical embedding of $E$ into its bidual, and $B_E$ will be the closed unit ball of $E$. We will denote by $FIN(E)$  the class of all finite dimensional subspaces of $E$ and $COFIN(E)$ will stand for the class of all finite codimensional closed subspaces of the space $E$.

A surjective mapping $T: E \to F$ is called \emph{a metric surjection} if $$\|T(x)\|_F=\inf\{\|y\|_E : T(y)=x \},$$ for all $x \in E$. As usual, a mapping $I : E \to F$ is called \emph{isometry} if $\|Ix\|_F = \|x\|_E$ for all $x \in E$. We will use the notation  $\overset 1 \twoheadrightarrow$ and  $\overset 1 \hookrightarrow$ to indicate a metric surjection or an isometry, respectively.
We also write $E\overset 1=F$ whenever $E$ and $F$ are isometrically isomorphic Banach spaces (i.e., there exist an surjective isometry $I : E \to F$).
For a Banach space $E$ with unit ball $B_E$, we call the mapping $Q_E : \ell_1(B_E) \overset 1 \twoheadrightarrow E$ given by
\begin{equation}\label{canonical quotient}Q_E\big((a_x)_{x\in B_E} \big)= \sum_{x\in B_E} a_x x\end{equation}
the \emph{canonical quotient mapping}.
Also, we consider the \emph{canonical embedding} $I_E: E\to \ell_\infty(B_{E'})$ given by \begin{equation}\label{canonical embedding}I_E(x)=\big(x'(x)\big)_{x'\in B_{E'}}.\end{equation}

A normed space $E$ has the \emph{$\lambda$-bounded approximation property} if there is a net $(T_{\eta})_\eta$ of finite rank operators in $\mathcal{L}(E,E)$ with norm bounded by $\lambda$ such that $T_\eta $ conveges to $Id_E$ (the identity operator on $E$) uniformly on compact subsets of $E$. A Banach space has the \emph{metric approximation property} if it has the 1-bounded approximation property.
\bigskip

For a normed space $E$, we will denote by $\otimes^n E$ the $n$-fold
tensor product of $E$.
For simplicity, $\os x$ will stand for the tensor the elementary tensor $x \otimes\cdots \otimes
x$.
The subspace of $\otimes^n E$ consisting of all
tensors of the form $\sum_{j=1}^r \lambda_j \os x_j$, where $\lambda_j$ is a scalar and $x_j\in E$ for all $j$, is called the
\emph{symmetric $n$-fold tensor product of $E$} and is denoted by $\os E$. When $E$ is a vector space over $\mathcal C$, the scalars are not needed in the previous expression. For simplicity, we will use the complex notation, although our results hold for real and complex spaces. 

Given a normed space $E$ and a continuous operator $T\colon E\to F$,
\emph{the symmetric $n$-tensor power of $T$} (or the \emph{tensor operator of $T$}) is the mapping from $\os
E$ to $\os F$ defined by
$$
\tsty \Big( \os T \Big)(\os x) = \os (Tx)
$$
on the elementary tensors and extended by linearity.

For an $n$-fold symmetric tensor $z\in \os E$, the \emph{symmetric projective norm} of $z$ is given by
$$
\pi_s(z)=\inf\left\{\sum_{j=1}^r \| x_j\|^n\right\},
$$
where the infimum is taken over
all the representations of $z$ of the form  $\sum_{j=1}^r \os x_j$.

On the other hand
\emph{the symmetric
injective norm} of $z$ is defined by
$$
\varepsilon_s(z)= \sup_{x'\in B_{E'}}\left|\sum_{j=1}^r
 x'(x_j)^n\right|,
$$
where $\sum_{j=1}^r \os x_j$ is any fixed representation of $z$.
For properties of these two classical norms ($\varepsilon_s$ and $\pi_s$) see \cite{Flo97}.

\medskip

Symmetric tensor products linearize homogeneous polynomials. Recall that a function $p\colon E\to \mathbb{K}$ is said to be a (continuous)
\emph{$n$-homogeneous polynomial} if there exists a
(continuous) symmetric $n$-linear form \allowbreak $$ A \colon
\underbrace{E\times\cdots\times E}_{{\rm n-times}} \to \mathbb{K}$$ such that
$p(x)= A(x,\ldots,x)$ for all $x\in E$.
In this case, $A$ is called the symmetric $n$-linear form associated to $p$.
Continuous
$n$-homogeneous polynomials are those bounded on the unit ball, and the norm of such $p$ is given by $$\|p\|=\sup_{\|x\|\le 1}|p(x)|.$$ If we denote
by ${\P}(E)$ the Banach space of all continuous
$n$-homogeneous polynomials on $E$ endowed endowed with the sup norm, we have the isometric identification \begin{equation}\label{dualidad-pol-tensor}{\P}(E) \overset 1 = \big(\os_{\pi_s} E\big)'.\end{equation}

\bigskip
We say that  $\alpha$ is an \emph{s-tensor norm  of order $n$} if $\alpha$ assigns to each normed space $E$ a norm $\alpha \big(\; . \;; \otimes^{n,s} E \big)$ on the $n$-fold symmetric tensor product $\otimes^{n,s} E$ such that
\begin{enumerate}
\item $\varepsilon_s \leq \alpha \leq \pi_s$ on $\otimes^{n,s} E$.
\item $\| \os T :   \os_\alpha E \to \os_\alpha F \| \leq \|T\|^n$ for each operator $T \in \mathcal{L}(E,F)$.
\end{enumerate}
Condition $(2)$ will be referred to as the \emph{``metric mapping property''}. We denote by ${\otimes}^{n,s}_\alpha E$ the tensor product ${\otimes}^{n,s}E$ endowed with the norm $\alpha \big(\; . \;; \otimes^{n,s} E \big)$, and we write
$\widetilde{\otimes}^{n,s}_\alpha E$ for its completion.

An s-tensor norm $\alpha$ is called \emph{finitely generated} if for every normed space $E$ and $z \in \otimes^{n,s} E$, we have: $$ \alpha (z, \otimes^{n,s}E) = \inf \{ \alpha(z, \otimes^{n,s}M) : M \in FIN(E),\; z \in \otimes^{n,s}M \}.$$
The norm $\alpha$ is called \emph{cofinitely generated} if for every normed space $E$ and $z \in \otimes^{n,s} E$, we have: $$ \alpha (z, \otimes^{n,s}E) = \sup \{ \alpha \big((\os Q_L^E) (z), \otimes^{n,s}E/L \big) : L \in COFIN(E)\}, $$
where $Q_L^E : E \overset 1 \twoheadrightarrow E/L$ is the canonical mapping.

If $\alpha$ is an s-tensor norm of order $n$, then the \emph{dual tensor norm $\alpha'$} is defined on FIN (the class of finite dimensional spaces) by
\begin{equation}\label{defi prima dim finita}   \os_{\alpha'} M :\overset 1 = \big( \os_{\alpha} M'\big)'\end{equation}
and on NORM (the class of normed spaces) by
$$ \alpha' ( z, \os E ) : = \inf \{ \alpha' (z, \os M ) : z \in \os M \},$$
the infimum being taken over all of finite dimensional subspaces $M$ of $E$ whose symmetric tensor product contains $z$. By definition, $\alpha'$ is always finitely generated.

Given a tensor norm $\alpha$ its \emph{``finite hull''} $\overrightarrow{\alpha}$ is defined by the following way. For $z \in \os E$, we set
$$\overrightarrow{\alpha}(z,\os E) := \inf \{ \alpha(z; \os M) : M \in FIN(E), z \in \os M  \}.$$
An important remark is in order: since $\alpha$ and $\alpha''$ coincide on finite dimensional spaces we have
$$
\overrightarrow{\alpha}(z;\os E) = \inf \{ \alpha''(z; \os M) : M \in FIN(E), z \in \os M  \} = \alpha''(z;\os E),
$$
where the second equality is due to the fact that dual norms are always finitely generated. Therefore,
\begin{equation} \label{bidual de una norma tensorial}
\overrightarrow{\alpha} = \alpha '';
\end{equation} and $\alpha=\alpha''$ if and only if $\alpha$ is finitely generated.

The \emph{``cofinite hull''} $\overleftarrow{\alpha}$ is given by
$$\overleftarrow{\alpha}(z;\os E) := \sup \{ \alpha\big( (\os Q_L^E) (z); \os E/L \big) : L \in COFIN(E)  \},$$
 where $Q_L^E: E \overset 1 \twoheadrightarrow E/L$ is the canonical quotient mapping.
Is not hard to see that the ``finite hull'' $\overrightarrow{\alpha}$ (the ``cofinite hull'' $\overleftarrow{\alpha}$) is the unique finitely generated s-tensor norm (cofinitely generated s-tensor norm) that coincides with $\alpha$ in finite dimensional spaces.
By the metric mapping property, it is enough to take cofinally many $M$ (or  $L$) in the definitions of the finite (or cofinite) hull. Using the metric mapping property again we have $$\overleftarrow{\alpha} \leq \alpha \leq \overrightarrow{\alpha}.$$

\bigskip
Since  any s-tensor norm satisfies $\alpha \leq \pi_s$, we have a dense inclusion $${\otimes}^{n,s}_\alpha  E \hookrightarrow  {\otimes}^{n,s}_\pi E. $$ As a consequence, any $p\in \big({\otimes}^{n,s}_\alpha  E\big)'$ identifies with a $n$-homogeneous polynomial on $E$. Different s-tensor norms $\alpha$ give rise, by this duality, to different classes of polynomials (see Section~\ref{seccion-polinomios}).

\section{The Lemmas}
In this section we will give in full detail the symmetric analogues to the five basic lemmas that appear in \cite[Section 13]{DefFlo93}. The first of them states that for normed spaces with the bounded approximation property, it is enough to check dominations between s-tensor norms on finite dimensional subspaces.

\begin{lemma}\label{Aprox lema}
\textbf{\emph{ (Approximation Lemma.)}} Let $\beta$ and $\gamma$ be s-tesnor norms, $E$ a normed space with the $\lambda$-bounded approximation property with constant and $c \geq 0$ such that
$$ \alpha \leq c \beta \;\; \mbox{on} \; \os M,$$
for cofinally many $M \in FIN(E)$. Then
$$\alpha \leq \lambda^n c \beta \;\; \mbox{on} \; \os E.$$

\end{lemma}
\begin{proof}Take $(T_{\eta})_\eta$  a net of finite rank operators with $\|T_\eta\|\le \lambda$ and $T_\eta x \to x$ for all $x \in E$. Fix $z\in \os E$ and take $\varepsilon >0$. Since the mapping $x\mapsto \otimes^n x$ is continuous from $E$ to $\otimes^{n,s}_\alpha E$, we have $\alpha(z-T_\eta (z),\os E )<\varepsilon$ for some $\eta$ large enough.
If we take $M\supset T_\eta(E)$ satisfying the hypothesis of the lemma, by the metric mapping property of the s-tensor $\beta$ we have
\begin{align*}
\alpha(z;\os E) & \leq  \alpha(z - \os T_\eta (z);\os E) + \alpha(\os T_\eta (z);\os E) \\
& \leq \varepsilon + \alpha(\os T_\eta (z);\os M) \\
& \leq \varepsilon + c \beta(\os T_\eta (z);\os M) \\
& \leq \varepsilon + c \|T_\eta: E \to M \|^n \beta(\os z;\os E) \\
& \leq \varepsilon + \lambda^n c  \beta(\os z;\os E).
\end{align*}
Since this holds for every $\varepsilon > 0$,  we have $\alpha(z;\os E) \leq \lambda^n c  \beta(z;\os E)$.
\end{proof}

\smallskip
Now we will devote our efforts to give a symmetric version of the Extension Lemma~\cite[6.7.]{DefFlo93}. A bilinear form $\varphi$ on $E\times F$ can be canonically extended to a bilinear form $\varphi^{\wedge}$ on $E\times F''$ and to a bilinear for $^\wedge \varphi$  on $E''\times F$ \cite[1.9]{DefFlo93}.
Aron and Berner showed in~\cite{AroBer78} (see also \cite{NachoSurvey}) how to extend continuous polynomials (and some holomorphic functions) defined on a Banach space $E$ to the bidual $E''$. The Aron-Berner extension can be seen as a symmetric (or polynomial) version of the canonical extensions $\varphi^{\wedge}$ and $^\wedge \varphi$.
In order to show that some holomorphic functions defined on the unit ball of $E$ can be extended to the ball of $E''$, Davie and Gamelin~\cite{DavieGamelin98} proved that this extension preserves the norm of the polynomial. If we look at the duality between polynomials and symmetric tensor products in \eqref{dualidad-pol-tensor}, Davie and Gamelin's result states that for $p$ in $( \widetilde{\otimes}^{n,s}_{\pi_s} E \big)'$, its Aron-Berner extension $AB(p)$ belongs to $( \widetilde{\otimes}^{n,s}_{\pi_s} E'' \big)'$, and has the same norm as $p$.
A natural question arises: if a polynomial $p$ belongs to $\big( \widetilde{\otimes}^{n,s}_\alpha
E \big)'$ for some s-tensor norm $\alpha$, does it Aron-Berner extension $AB(p)$ belong to $\big( \widetilde{\otimes}^{n,s}_\alpha
E'' \big)'$? And what about their norms?  The answer is given in the following result.

.

\begin{lemma} \label{extension lemma}
\textbf{\emph{(Extension Lemma.)}} Let $\alpha$ be a finitely generated s-tensor norm and $p \in \big( \widetilde{\otimes}^{n,s}_\alpha E \big)'$ a
polynomial. The Aron-Berner extension $AB(p)$ of $p$ to the bidual
$E''$ belongs to $\big( \widetilde{\otimes}^{n,s}_\alpha
E'' \big)'$ and
$$ \|p\|_{\big( \widetilde{\otimes}^{n,s}_\alpha E \big)'} = \|AB(p)\|_{\big( \widetilde{\otimes}^{n,s}_\alpha E'' \big)'}.$$
\end{lemma}

In \cite{Car-GalPrims} we used ultrapower techniques to prove this lemma. We will give here a direct using standar tools. First, we recall the following.
\begin{theorem}\label{localreflexivity} \textbf{\emph{(The Principle of Local Reflexivity.)}}
For each $M \in FIN(E'')$, $N \in FIN(E')$ and $\varepsilon > 0$, there
exists an operator $R \in \mathcal{L}(M,E)$ such that
\begin{enumerate}
\item  R is an $\varepsilon$-isometry; that is, $(1-\varepsilon) \|x''\| \leq \|R(x'')\| \leq (1+ \varepsilon ) \|x''\|;$
\item $R(x'')=x''$ for every $x'' \in M \cap E$;
\item $ x'( R (x'') ) = x''(x')$ for $x'' \in M$ and $x' \in
N$.
\end{enumerate}
\end{theorem}

Let $A$ be a symmetric multilinear form. For each fixed $j$, $1 \leq j \leq n$, $x_1, \dots, x_{j-1} \in
E$, and $x_j'', x_{j+1}'', \dots x_n'' \in E''$, it is easy to see that
$$\overline{A}(x_1, \dots, x_{j-1}, x_j'', x_{j+1}'', \dots,
x_n'') = \lim_{\alpha_j} \overline{A}(x_1, \dots,
x_{j-1},x_{{\alpha}_j}^{(j)}, x_{j+1}'', \dots, x_{n}''),$$ where
$\overline A$ is the iterated extension of $A$ to $E''$ and $(x_{{\alpha}_j}^{(j)}) \in E$ such that $w^*-\lim_{\alpha_j} x_{{\alpha}_j}^{(j)} = x_j''$.

Now, we will imitate the procedure used by Davie and
Gamelin in \cite{DavieGamelin98}. From now on $A$ will be the symmetric $n$-linear form associated to
$p$. We have the following lemma.

\begin{lemma}\label{diferentes indices} Let $M \in FIN(E'')$ and $x_1'', \dots, x_r'' \in M$. For a given natural number $m$, and $\varepsilon >
0$ there exist operators $R_1, \dots, R_m \in \mathcal{L}(M,E)$
with norm less or equal to $1 + \varepsilon$ such that
\begin{equation}\label{desigualdad lema}  A(x_k'', \dots, x_k'') = A( R_{i_1}x_k'', \dots,
R_{i_n}x_k'' )\end{equation} for every $i_1, \dots, i_n$
distinct indices between $1$ and $m$ and every $k = 1 \dots r$.
\end{lemma}

\begin{proof}
Since $A$ is symmetric, in order  to prove the Lemma it suffices
to obtain \eqref{desigualdad lema} for $i_1 < \dots < i_n$. We
will select the operator $R_1, \dots, R_m$ inductively by the
following procedure: let $N_1 := [\overline{A}(\cdot, x_k'', \dots, x_k '') : 1 \leq k \leq r ] \subset E'$, by the Principle of Local Reflexivity~\ref{localreflexivity}, there exist an operator $R_2 \in
\mathcal{L}(M,E)$ with norm less or equal to $1+\varepsilon$ such that $x'( R_1 (x'') ) = x''(x')$ for $x'' \in M$ and $x' \in
N_1$.
In particular,
$$\overline{A}(R_1(x_k''), x_k'', \dots, x_k '') = x_k''(\overline{A}(\cdot, x_k'', \dots, x_k '')) = \overline{A}(x_k'', x_k'', \dots, x_k ''),$$
for every $1 \leq k \leq r$.

Now, let $N_2:=[\overline{A}(\; \cdot \; , x_k'', \dots, x_k '') : 1 \leq k \leq r ] \oplus [\overline{A}(R_1 x_k'', \; \cdot \;,  x_k'', \dots, x_k '') : 1 \leq k \leq r] \subset E'$.
Again, by the Principle of Local Reflexivity~\ref{localreflexivity}, there exist an operator $R_1 \in
\mathcal{L}(M,E)$ with norm less or equal than $1+\varepsilon$ such that
$$\overline{A}(R_2(x_k''), x_k'', \dots, x_k '') = \overline{A}(x_k'', x_k'', \dots, x_k '')$$ and $$\overline{A}(R_1 (x_k''), R_2 (x_k''), x_k'' \dots, x_k '') = \overline{A}(x_k'', x_k'', \dots, x_k ''),
$$ for all $k=1 \dots, r$.
Proceeding in this way we obtain the stated result.
\end{proof}

\begin{lemma} \label{Aprox}  Let $M \in FIN(E'')$ and $x_1'', \dots, x_r'' \in M$, $p: E \to
\mathbb{K}$ a continuous $n$-homogeneous polynomial and $\varepsilon > 0$. There exist $m \in \mathbb{N}$ and operators $(R_i)_{1\leq i \leq m}$ in
$\mathcal{L}(M,E)$ with norm less or equal than $1+\varepsilon$, satisfying
$$ \big| \sum_{k=1}^r AB(p) \big( x_k''\big) - \sum_{k=1}^r p \big( \frac{1}{m} \sum_{i = 1}^m R_i x_k'' \big) \big| < \varepsilon. $$
\end{lemma}

\begin{proof} For
$\varepsilon > 0$, fix $m$ large enough and choose $R_1, \dots, R_m$ as
in the previous Lemma, such that
\begin{equation}\label{ecuacionigualdad}
 A( R_{i_1}x_k'', \dots, R_{i_n}x_k'' ) = \overline{A}(x_k'', \dots,
x_k'')
\end{equation}
for every $i_1, \dots, i_n$ distinct indices between $1$ and $m$
and every $k = 1 \dots r$. We have
\begin{align*}
\big| AB(p)(x_k'') - p(\frac{1}{m} \sum_{i = 1}^m R_i x_k'') \big| & = \big| \frac{1}{m^n}
\sum_{i_1, \dots, i_n = 1}^m [\overline{A}(x_k'', \dots, x_k'') -
A(R_{{i_1}}x_k'', \dots, R_{{i_n}}x_k'' )] \big| \\
& \leq \big| \Sigma_1^k \big| + | \Sigma_2^k |,
\end{align*}
where $\Sigma_1^k$ is the sum over the $n$-tuples of non-repeated
indices, which is zero by (\ref{ecuacionigualdad}) and $\Sigma_2^k$ is
the sum over the remaining indices. It is easy to show that there
are exactly $m^n - \prod_{j=0}^{n-1} (m-j)$ summands in
$\Sigma_2^k$, each bounded by a constant $C>0$, which we can
assume independent of $k$. Thus we have
$$\big| \Sigma_2^k  | \leq \frac{1}{m^n} \big (m^n - \prod_{j=0}^{n-1}
(m-j) \big)C= \big[ 1 - (1-\frac{1}{m}) \dots (1 - \frac{n-1}{m})
\big] C.$$ Taking $m$ sufficiently large this is less than
$\varepsilon/{r}$.
\end{proof}

Now we are ready to prove the Extension Lemma~\ref{extension lemma}.

\begin{proof} (of Lemma~\ref{extension lemma})
Let $w \in \otimes^{n,s} M$, where $M \in FIN(E'')$. Since $\alpha$ is finitely generated, we only have to check
that $$| \langle AB(p),w \rangle| \leq \|p\|_{\big(
\widetilde{\otimes}^{n,s}_\alpha E \big)'} \: \alpha(w,
\otimes^{n,s}M).$$ Now write $w = \sum_{k=1}^r \otimes^n x_k''$ with $x_k'' \in
M$. Given $\varepsilon > 0$, by Lemma~\ref{Aprox}, there exist $m \in \mathbb{N}$ and operators $(R_i)_{1\leq i \leq m}$ with
$\|R_i\|_{\mathcal{L}(M,E)} \leq 1+ \varepsilon$ such that $$ \big|
\sum_{k=1}^r AB(p) \big( x_k''\big) - \sum_{k=1}^r p \big(
\frac{1}{m} \sum_{i=1}^m R_i x_k'' \big) \big| < \varepsilon.$$
Therefore,
\begin{align*}
\big| \langle AB(p),w \rangle \big|  & = \big| \sum_{k=1}^r AB(p)
\big( x_k''\big) \big| \leq \big| \sum_{k=1}^r AB(p) \big(
x_k''\big) - \sum_{k=1}^r p \big( \frac{1}{m} \sum_{i=1}^m R_i
x_k'' \big) \big| + \big| \sum_{k=1}^r p \big( \frac{1}{m} \sum_{i=1}^m R_i x_k'' \big) \big| \\
& \leq \varepsilon + \big| \langle p, \sum_{k=1}^r \otimes^n
\frac{1}{m} \sum_{i=1}^m R_i x_k'' \rangle \big| \\
& \leq \varepsilon + \|p\|_{\big( \widetilde{\otimes}^{n,s}_\alpha
E \big)'} \alpha( \sum_{k=1}^r \otimes^n \frac{1}{m} \sum_{i=1}^m R_i x_k'' \: ; \; \otimes^{n,s}E) \\
& \leq \varepsilon + \|p\|_{\big( \widetilde{\otimes}^{n,s}_\alpha
E \big)'} \alpha\big( (\otimes^{n,s}R) (\sum_{k=1}^r  x_k'') \:
; \; \otimes^{n,s}E \big),
\end{align*}
where $R = \frac{1}{m} \sum_{i=1}^m R_i$ (note that
$\|R\|_{\mathcal{L}(M,E)} \leq 1 + \varepsilon$ since each
$\|R_i\|_{\mathcal{L}(M,E)} \leq 1 + \varepsilon$). By the metric mapping
property of $\alpha$ and the previous inequality we get
$$ \big| \langle AB(p),w \rangle  \big| \leq \varepsilon + (1 + \varepsilon)^n \|p\|_{\big( \widetilde{\otimes}^{n,s}_\alpha
E \big)'} \alpha( \sum_{k=1}^r \otimes^n x_k'' \: ; \;
\otimes^{n,s}M),$$ which ends the proof.
\end{proof}

As a consequence of the Extension Lemma \ref{extension lemma} we also obtain a symmetric version of \cite[Lemma~13.3]{DefFlo93}, which shows that there is a natural isometric embedding from the symmetric tensor product of a Banach space and that of its bidual. This lemma
was proved in~\cite{Car-GalPrims} for  finitely generated s-tensor norms.

\begin{lemma}\label{embedding}
\textbf{\emph{(Embedding Lemma.)}} If $\alpha$ is a finitely or cofinitely generated tensor norm, then the natural mapping
\begin{align*}
&  \otimes^{n,s}\K_E: \otimes^{n,s}_{\alpha} E \longrightarrow \otimes^{n,s}_{\alpha} E''
\end{align*}
is an isometry for every normed space $E$.
\end{lemma}
\begin{proof}
If $z \in \os E$, by the metric mapping property we have
$$\alpha(\os \K_E (z);\os E'') \leq \alpha(z;\os E).$$
Suppose $\alpha$ is finitely generated and let $p$ a norm one polynomial in $(\os_{\alpha} E)'$ such that $\alpha(z;\os E) =  \langle p, z \rangle$. Now notice that $ \langle p, z \rangle = \langle AB(p), \otimes^{n,s} \K_E z \rangle$ which, by the Extension Lemma~\ref{extension lemma}, is less or equal than $\alpha(\os \K_E (z);\os E'')$. This shows the reverse inequality for finitely generated tensor norms.

Suppose now that $\alpha$ is cofinitely generated and let $L \in COFIN(E)$. Then $L^{00}$ (the biannihilator in $E''$) is in $COFIN(E'')$ and the mapping
$$ \K_{E/L} : E/L \to (E/L)''=E''/{L^{00}}$$ is an isometric isomorphism. Moreover, we have $Q_{L^{00}}^{E''} \circ \K_E = \K_{E/L} \circ Q_{L}^{E}$.

Thus,
\begin{align*}
\alpha (\os Q_{L}^E (z) ; \os E/L) & = \alpha (\os (\K_{F/L} \circ Q_{L}^{E}) (z) ; \os (E/L)'') \\
& =  \alpha ((\os Q_{L^{00}}^{E''} \circ  \os  \K_E) (z) ; \os E''/{L^{00}}) \\
& \leq \alpha (\os  \K_E (z), \os E'').
\end{align*}
If we take supremum over all $L \in COFIN(E)$ we obtain the desired inequality.
\end{proof}

Since $E$ and its completion $\widetilde{E}$ have the same bidual, the Embedding Lemma~\ref{embedding} shows that finitely generated and cofinitely generated s-tensor norms respect dense subspaces. More precisely, we have the following.

\begin{corollary}
Let $\alpha$ be a finitely or cofinitely generated s-tensor norm, $E$ a normed space and $\widetilde{E}$ its completion.  Then,
$$\os_{\alpha} E \to  \os_{\alpha} \widetilde{E}$$
is an isometric and dense embedding.
\end{corollary}

We obtain as a direct consequence the symmetric version of the Density lemma \cite[Lemma 13.4.]{DefFlo93}.

\begin{lemma}\label{density}
\textbf{\emph{(Density Lemma.)}} Let $\alpha$ be a finitely or cofinitely generated tensor norm, $E$ a normed space and $E_0$ a dense subspace of $E$. If $p$ is an $n$-homogeneous continuous polynomial such that
$$ p|_{\os E_0} \in (\os_\alpha E_0)',$$
then $p \in (\os_{\alpha} E )'$ and  $\|p\|_{(\os_{\alpha} E )'} = \|p\|_{(\os_{\alpha} E_0 )'}$.
\end{lemma}

\smallskip

Before we state the fifth lemma, we need some definitions.
For $1\leq p \leq \infty$ and $1 \leq \lambda < \infty$ a normed space $E$ is called an $\mathcal{L}_{p,\lambda}^g$\emph{-space}, if for each $M \in FIN(E)$ and $\varepsilon >0$ there are $R \in \mathcal{L}(M,\ell_p^m)$ and $S \in \mathcal{L}(\ell_p^m,E)$ for some $m \in \mathbb{N}$ factoring the embedding $I_M^E$ such that $\|S\| \|R\| \leq \lambda + \varepsilon$:
\begin{equation} \label{facto}
\xymatrix{ M  \ar@{^{(}->}[rr]^{I_E^M} \ar[rd]^{R} & & {E} \\
& {\ell_p^m} \ar[ur]^{S} & }.
\end{equation}
$E$ is called an $\mathcal{L}_{p}^g$\emph{-space} if it is an $\mathcal{L}_{p,\lambda}^g$-space for some $\lambda \geq 1$.
Loosely speaking, $\mathcal{L}_{p}^g$-spaces share many properties of $\ell_p$, since they locally look like $\ell_p^m$. The spaces
$C(K)$ and $L_\infty(\mu)$ are $\mathcal{L}_{\infty,1}^g$-spaces, while $L_p(\mu)$ are $\mathcal{L}_{p,1}^g$-spaces.
For more information and properties of $\mathcal{L}_{p}^g$-spaces see \cite[Section 23]{DefFlo93}.

Now we will state and prove our fifth basic lemma.
\begin{lemma}\label{local technique}
\textbf{\emph{($\mathcal{L}_p$-Local Technique Lemma.)}} Let $\alpha$ and $\beta$ be s-tensor norm and $c \geq 0$ such that
$$ \alpha \leq c \beta \;\;\; \mbox{on}  \;\;\; \os{\ell_p^m},$$
for every $m \in \mathbb{N}$. If $E$ is an $\mathcal{L}_{p,\lambda}^g$ normed space then
$$ \alpha \leq \lambda^n c \overrightarrow{\beta} \;\;\; \mbox{on}  \;\;\; \os{E}.$$

\end{lemma}

\begin{proof}
For $M \in FIN(E)$, we take a factorization as in (\ref{facto}) with $\|R\| \|S\| \leq \lambda (1 + \varepsilon)$.
Then, for every $z \in \os M$ we have
\begin{align*}
\alpha(z;\os M) & = \alpha(\os (S \circ R) (z), \os M) \leq \|S\| \alpha(\os R (z), \os \ell_p^m) \\
& \leq \|S\|^n c \beta(\os R (z), \os \ell_p^m) \leq c \|S\|^n \|R\|^n \beta(z;\os M).
\end{align*}
Taking infimum over all finite dimensional subspaces $M$ such that  $z \in \os M$, we obtain $$\alpha \leq \lambda^n c \overrightarrow{\beta},$$as desired.
\end{proof}

\section{Applications to the metric theory of symmetric tensor products}

In this section we present applications of the five basic lemmas to the study of symmetric tensor norms, specifically to their metric properties. 
The first application of the lemmas that we get relates the finite hull of an s-tensor norm with its cofinite hull on $\os E$ when $E$ has the bounded approximation property.
\begin{proposition}\label{coinciden con m.a.p.}
Let $\alpha$ be an s-tensor norm and $E$ be a normed space with the $\lambda$-bounded approximation property. Then
$$ \overleftarrow{\alpha} \leq \alpha \leq \overrightarrow{\alpha} \leq \lambda^n \overleftarrow{\alpha} \; \; \mbox{on} \; \; \os E.$$
In particular, $\overleftarrow{\alpha} = \alpha = \overrightarrow{\alpha}$ on $\os E$ if $E$ has the metric approximation property.
\end{proposition}

\begin{proof}
Is a direct consequence of the Approximation Lemma \ref{Aprox lema} and the fact that $\overleftarrow{\alpha} = \alpha=  \overrightarrow{\alpha}$ on $\os M$ for every $M \in FIN(E)$
\end{proof}
This proposition together with the Embedding Lemma~\ref{extension lemma} give the following corollary, which should be compared to the Embedding Lemma~\ref{embedding}. Note that the assumptions on the s-tensor norm $\alpha$ in the Embedding Lemma are now substituted  by assumptions on the normed space $E$.

\begin{corollary}
Let $\alpha$ be an s-tensor norm and $E$ be a normed space with the metric approximation property. Then \begin{align*}
&  \otimes^{n,s}\K_E: \otimes^{n,s}_{\alpha} E \longrightarrow \otimes^{n,s}_{\alpha} E''
\end{align*}
is an isometry.
\end{corollary}

\begin{proof}
If $z \in \os E$, by the metric mapping property
$$\alpha (\os \kappa_E z; \os E'') \leq  \alpha (\os z; \os E).$$
On the other hand, since $E$ has the metric mapping property,  Proposition~\ref{coinciden con m.a.p.} asserts that $\alpha = \overleftarrow{\alpha}$  on $\os E$. We then have
$$
\alpha (\os z; \os E) = \overleftarrow{\alpha} (\os z; \os E) = \overleftarrow{\alpha} (\os \kappa_E z; \os E'') \leq \alpha (\os \kappa_E z; \os E''),
$$
where the second equality is due to the Embedding Lemma~\ref{extension lemma} applied to the cofinitely generated s-tensor norm $\overleftarrow{\alpha}$.
\end{proof}

For a finite dimensional space $M$ we always have the isometric isomorphism
\begin{equation}\label{prima para dim finita}
\os_{\alpha}M' \overset 1 = (\os_{\alpha'} M)'.
\end{equation}
The next theorem and its corollary show the behaviour of the mappings in \eqref{defi prima dim finita} and \eqref{prima para dim finita} in the  infinite dimensional framework.

\begin{theorem}\label{duality theorem}
\textbf{\emph{(Duality Theorem.)}}
Let $\alpha$ be an s-tensor norm. For every normed space $E$ the following natural mappings are isometries:
\begin{equation} \label{duality 1}
\os_{\overleftarrow{\alpha}}E \hookrightarrow (\os_{\alpha '} E')',
\end{equation}
\begin{equation}\label{duality 2}
\os_{\overleftarrow{\alpha}}E' \hookrightarrow (\os_{\alpha '} E)'.
\end{equation}

\end{theorem}

\begin{proof}
Let us prove that the first mapping is an isometry. Observe that $$FIN(E')=\{L^0 : L \in COFIN(E) \}.$$
Now, by the duality relation for finite dimensional spaces (\ref{prima para dim finita}) we obtain
\begin{align*}
\overleftarrow{\alpha}(z; \os E) & = \sup_{L \in COFIN(E)} \alpha(Q_{L}^E (z) ; \os E/L ) \\
& = \sup_{L \in COFIN(E)} \sup \{ \langle Q_{L}^E (z),u \rangle : \alpha' (u; \os L^0) \leq 1  \} \\
& = \sup \{ \langle Q_{L}^E (z),u \rangle : \overrightarrow{\alpha '} (u; \os E') \leq 1  \} \\
& = \sup \{ \langle Q_{L}^E (z),u \rangle : \alpha ' (u; \os E') \leq 1  \},
\end{align*}
and this shows (\ref{duality 1}).

For the second mapping, note that the following diagram commutes
\begin{equation}\label{diamaccess}
\xymatrix{
{\os_{\overleftarrow{\alpha}} E'} \ar@{^{(}->}[rr]^{1}\ar[rrd] & & {(\os_{\alpha'} E'')'} & {\ni AB(p)}\\
& & {(\os_{\alpha '} E)'} \ar@{^{(}->}[u]^{1} & {\ni p} \ar@{~>}[u]
}.
\end{equation}
Then, the Extension Lemma~\ref{extension lemma} gives the isometry $\os_{\overleftarrow{\alpha}}E' \hookrightarrow (\os_{\alpha '} E)'$, so we have $(\ref{duality 2}$).
\end{proof}

\begin{corollary}
Let $\alpha$ be an s-tensor norm. For every normed space the mappings
\begin{eqnarray}
\label{sin prima} \os_{\alpha}E &\hookrightarrow& (\os_{\alpha '} E')' \\ \label{con prima} \os_{\alpha}E' &\hookrightarrow& (\os_{\alpha '} E)'
\end{eqnarray}
are continuous and has norm one.
If $E$ (or equivalently, $E'$) has the metric approximation property or $\alpha$ is cofinitely generated, then the mapping are isometries.
\end{corollary}

\begin{proof}
Since $\overleftarrow{\alpha} \leq \alpha$, continuity and that the norm of both mappings is one follow from the Duality Theorem~\ref{duality theorem}. If $E$ has the metric approximation property, then $\overleftarrow{\alpha} = \alpha$ on $\os E$ and on $\os E'$ by Proposition \ref{coinciden con m.a.p.}, so the conclusion follows again from the Duality Theorem.
\end{proof}

The isometry \eqref{con prima} for the case of $E'$ having the metric approximation property  can also be obtained also from \cite[Corrollary 5.2 and Proposition 7.5]{Flo01}. Note also that if $E$ (respectively, $E'$) has the $\lambda$-approximation property, then the mapping \eqref{sin prima} (respectively, \eqref{con prima}) is an isomorphism onto its range.

\bigskip

Let $\alpha$ be an s-tensor norm of order $n$.
We will say that $\alpha$ is \emph{projective} if, for every metric surjection $Q: E \overset 1 \twoheadrightarrow F$, the tensor product operator
$$ \os Q : \os_\alpha E \to \os F$$
is also a metric surjection. 
On the other hand we will say that $\alpha$ is \emph{injective} if, for every  $I: E \overset 1 \hookrightarrow F$ isometric embedding, the tensor product operator
$$ \os I : \os_\alpha E \to \os_\alpha F,$$
is an isometric embedding.

The two extreme s-tensor norms, $\pi_s$ and $\varepsilon_s$, are examples of the last two definition: $\pi_s$ is projective and $\varepsilon_s$ is injective.

We will now define the projective and injective associates of an s-tensor norm. \emph{The projective associate of} $\alpha$, denoted by $\s \alpha /$, will be the (unique) smallest projective s-tensor norm greater than $\alpha$. We can explicitly define it as
$$ \otimes^{n,s} Q_E \colon \otimes^{n,s}_\alpha \ell_1(E) \overset 1 \twoheadrightarrow  \otimes^{n,s}_{\s \alpha /}  E,$$  where $Q_E : \ell_1(B_E) \twoheadrightarrow E$ is the canonical quotient mapping defined in (\ref{canonical quotient}). That the s-tensor norm so defined is the smallest projective s-tensor norm greater than $\alpha$ follows as in  \cite[Theorem 20.6.]{DefFlo93}.

\emph{The injective associate of }$\alpha$, denote by $ / \alpha \s $, will be the (unique) greatest injective s-tensor norm smaller than $\alpha$.  As in \cite[Theorem 20.7.]{DefFlo93} we can describe it explicitly as
$$ \otimes^{n,s} I_E \colon  \otimes^{n,s}_{/ \alpha \s} E  \overset 1 \hookrightarrow  \otimes^{n,s}_\alpha \ell_{\infty}(B_{E'}),$$  where $I_E$ is the canonical embedding (\ref{canonical embedding}).

An s-tensor norm that appears in the literature that comes from this construction is the norm $\eta$ \cite{KirRy98,Carando99}, which coincides with $/ \pi_s \s$. This norm is the predual s-tensor norm of the ideal of extendible polynomials $\P_e$ (a definition is given in the next section).

The next result shows that an s-tensor norm coincides with its projective associate on the symmetric tensor product of  $\ell_1(I)$, where $I$ is any index set.
\begin{proposition} \label{coindicen en l1}
Let $\alpha$ be an s-tensor norm, then
$$\alpha = \s \alpha / \; \text{on} \; \os \ell_1(I),$$
for every index set $I$.
\end{proposition}

\begin{proof}
Let $\Q : \ell_1\big(B_{\ell_1(I)} \big) \overset 1 \twoheadrightarrow \ell_1(I)$ the natural quotient mapping.
Since $\ell_1(I)$ is projective then there is a lifting $T : \ell_1(I) \to \ell_1\big( B_{\ell_1(I)} \big)$ of $id_{\ell_1(I)}$ (i.e. $Q \circ T = id_{\ell_1(I)}$) having norm less or equal than $\|T\| \leq 1 + \varepsilon$.
Thus, by the diagram
\begin{equation*}
\xymatrix{
{\os_{\alpha}\ell_1(I)}  \ar@{-->}[rr]^{id} \ar[rrd]^{\os T}  & & {\os_{\s \alpha /} \ell_1(I)} \\
& & {\os \ell_1\big(B_{\ell_1(I)}\big)} \ar@{->>}[u]^{\os Q}
},
\end{equation*}
we have $\s \alpha / \leq (q + \varepsilon) \alpha$. Since $\alpha \leq \s \alpha /$ always holds, we have the desired equality.
\end{proof}

A Banach space space $E$ is called \emph{injective} if for every Banach space $F$, every subspace $G \subset F$ and every $T \in \mathcal{L}(G,E)$ there is an extension $\widetilde{T} \in \mathcal{L}(F,E)$ of $T$. The space $E$ has the \emph{$\lambda$-extension property} if there is a constant $\lambda \geq 1$ such that $\|\widetilde{T}\| \leq \lambda \|T\|$.
It is easy to prove that every Banach space with the $\lambda$-extension property is $\lambda$-complemented in $\ell_{\infty}(B_{E'})$. Using this fact we therefore have the following.

\begin{proposition}\label{coindicen en inft}
Let $\alpha$ be an s-tensor norm and $E$ be a Banach space with the $\lambda$-extension property, then
$$ / \alpha \s \leq \alpha \leq \lambda^n / \alpha \s \; \text{on} \; \os E.$$
In particular $$\alpha = / \alpha \s \; \text{on} \; \os \ell_\infty(I),$$
for every index set $I$.
\end{proposition}

A particular but crucial case of Proposition~\ref{coindicen en l1} and Proposition~\ref{coindicen en inft} is obtained with $I$ a finite set. In this we obtain for every s-tensor norm $\alpha$ and $m \in \mathbb{N}$,
$$ \alpha = \s \alpha / \; \text{on} \; \os \ell_1^m,$$
$$ \alpha  =  / \alpha \s \; \text{on} \; \os \ell_\infty^m .$$
The previous equalities allow us to use  $\mathcal{L}_p$-Local Technique Lemma~\ref{local technique} to give the following.

\begin{corollary} \label{Corolario con las normas en L_p spaces}
Let $\alpha$ an s-tensor norm
\begin{enumerate}
\item If $E$ is $\mathcal{L}_{1,\lambda}^g$-space, then
$$ \alpha \leq \s \alpha / \leq \lambda^n \overrightarrow{\alpha }\; \; \; \mbox{on}  \; \;\os E.$$
\item If $E$ is $\mathcal{L}_{\infty,\lambda}^g$-space, then
$$ \alpha \leq  / \alpha \s  \leq \lambda^n \overrightarrow{\alpha}  \; \; \; \mbox{on}  \; \; \os E.$$
\end{enumerate}
\end{corollary}

The next result show the relation between finite hulls, cofinite hulls, projective associates, injective associates and duality.

\begin{proposition} \label{relaciones felchita con palito}
For an s-tensor norm $\alpha$ we have the following relations:
\begin{enumerate}
\item $/ \alpha \s = / (\overrightarrow{\alpha}) \s = \overrightarrow{/ \alpha \s}$,
\item $/ \alpha \s = / (\overleftarrow{\alpha}) \s = \overleftarrow{/ \alpha \s}$,
\item $\s \alpha / = \s (\overrightarrow{\alpha}) / = \overrightarrow{\s \alpha /}$,
\item $\s \alpha / = \s (\overleftarrow{\alpha}) /$,
\item $(\s \alpha /)' = / \alpha' \s$ and $(/ \alpha \s)' = \s \alpha' /$.
\end{enumerate}
\end{proposition}

It is important to remark that the identity $\s \overleftarrow{\alpha}/  = \overleftarrow{\s \alpha /}$ fails to hold in general.
To see this, notice that $\overleftarrow{\pi_s} = \pi_s$ on $\os \ell_{1}^m$. Then, by Lemma~\ref{lemma con palitos} we have
$\s \overleftarrow{\pi_s} / =  \s \pi_s / = \pi_s$ (since $\pi_s$ is projective).
But $\pi_s$ is not cofinitely generated \cite[2.5.]{Flo01}. Thus,
$$\s \overleftarrow{\pi_s} / =  \s \pi_s / = \pi_s \neq \overleftarrow{\pi_s} =\overleftarrow{ \s \pi_s /}.$$
To prove Proposition~\ref{relaciones felchita con palito} we  need the following lemma.

\begin{lemma} \label{lemma con palitos}
Let $\alpha$ and $\beta$ be s-tensor norms.
\begin{enumerate}
\item The equality $\alpha = \beta$ holds on $\os\ell_1^m$ for every $m \in \mathbb{N}$ if and only if
$ \s \alpha / = \s \beta /$.
\item The equality $\alpha = \beta$ holds on $\os \ell_{\infty}^m $ for every $m \in \mathbb{N}$,
if and only if $ / \alpha \s = / \beta \s$.
\end{enumerate}
\end{lemma}

\begin{proof}
$(1)$ Suppose that $ \os_{\alpha} \ell_1^m \overset 1 = \os_{ \beta } \ell_1^m$ for every $m$.

If $E$ is a normed space and  $Q_E : \ell_1(B_E) \twoheadrightarrow E$ is the canonical quotient mapping defined in equation (\ref{canonical quotient}), we have
$$ \otimes^{n,s} Q_E \colon \otimes^{n,s}_\alpha \ell_1(E) \overset 1 \twoheadrightarrow  \otimes^{n,s}_{\s \alpha /}  E,$$
$$ \otimes^{n,s} Q_E \colon \otimes^{n,s}_{\beta} \ell_1(E) \overset 1 \twoheadrightarrow  \otimes^{n,s}_{\s \beta /}  E.$$
Since $\ell_1(B_E)$ has the metric approximation property,  by the $\mathcal{L}_p$-Local Technique Lemma~\ref{local technique} and Propositon~\ref{coinciden con m.a.p.} we have  $\alpha = \beta$ on $\os \ell_1(B_E)$. As a consequence,  we have$$ \s \alpha / = \s \beta / \; \text{on} \; \os E.$$
The converse is a direct consequence of Proposition~\ref{coindicen en l1}.

The proof in $(2)$ is similar. Suppose $\alpha = \beta$  on $\os \ell_{\infty}^m $ for every $m$. Again by the $\mathcal{L}_p$-Local Technique Lemma~\ref{local technique} and Propositon~\ref{coinciden con m.a.p.}, we have $\alpha = \beta$ on $\ell_\infty(B_{E'})$. To finish the proof we just use the isometric embeddings
$$ \otimes^{n,s} I_E \colon  \otimes^{n,s}_{/ \alpha \s} E  \overset 1 \hookrightarrow  \otimes^{n,s}_\alpha \ell_{\infty}(B_{E'}),$$
$$ \otimes^{n,s} I_E \colon  \otimes^{n,s}_{/ \beta \s} E  \overset 1 \hookrightarrow  \otimes^{n,s}_\beta \ell_{\infty}(B_{E'}).$$
The converse follows from Proposition~\ref{coindicen en inft}.
\end{proof}

Now we are ready to prove Proposition~\ref{relaciones felchita con palito}.
\begin{proof} (of Proposition~\ref{relaciones felchita con palito})
\\
$(1)$
Since $\alpha = \overrightarrow{\alpha}$ on $\os \ell_\infty^m$, for every $m$, by the Lemma~\ref{lemma con palitos} we have $$/ \alpha \s = / (\overrightarrow{\alpha}) \s \; \text{on} \; \os E.$$

To prove that $/ \alpha \s = \overrightarrow{/ \alpha \s}$, we first note that if $z \in \os M$ with $M \in FIN(E)$,  $/ \alpha \s$ being  injective we have
$$/ \alpha \s (z; \os M) = / \alpha \s (z; \os E).$$
As a consequence,
\begin{eqnarray*}\overrightarrow{/ \alpha \s }(z;\os E) &= &\inf \{ / \alpha \s(z;\os M) : z \in \os M \; \text{and} \; M \in FIN(E) \}\\ & = &
\inf \{ / \alpha \s(z;\os E) \} \\ &= & / \alpha \s(z;\os E).\end{eqnarray*}

\smallskip
$(2)$ Since  $\alpha = \overleftarrow{\alpha}$ on $\os \ell_\infty^m$ for every $m$, the equality  \begin{equation}\label{1radesigde2}/ \alpha \s = / (\overleftarrow{\alpha}) \s\end{equation} on $\os E$ follows Lemma~\ref{lemma con palitos}.
On the other hand,    Proposition~\ref{coinciden con m.a.p.} gives $ \overleftarrow{/ \alpha \s} \leq / \alpha \s$.
To show the reverse inequality, note that $$/ \alpha \s = / ( / \alpha \s ) \s =  / (\overleftarrow{/ \alpha \s}) \s,$$  where the second equality is just \eqref{1radesigde2} applied to $/ \alpha \s$.
Since by definition of the injective associate we have $/ \mu \s \leq \mu$ for every s-tensor norms $\mu$, taking $\mu=\overleftarrow{/ \alpha \s}$  we get
$ / (\overleftarrow{/ \alpha \s}) \s \leq \overleftarrow{/ \alpha \s},$ which gives de desired inequality.

\smallskip
$(3)$
The equality $\s \alpha / = \s (\overrightarrow{\alpha}) /$ is again a consequence of Lemma~\ref{lemma con palitos}.
Since $\ell_1(B_E)$ has the metric approximation property we therefore have $\alpha = \s \alpha / = \overrightarrow{\s \alpha /}$ on $\os \ell_1$.
Thus, for each element $z \in \os E$ and each $\varepsilon>0$ there is an $M \in FIN (\ell_1(B_E))$ and a $w \in \os M$ with $\os Q_E (w) = z$ and $$ \alpha (w; \os M) \leq (1+ \varepsilon) \s \alpha /(z;\os E).$$
Therefore,
\begin{eqnarray*} \s \alpha /(z;\os E) & \leq & \overrightarrow{\s \alpha /}(z; \os E) \\ & \leq & \s \alpha /(z; \os Q_E(M)) \leq \s \alpha /(w; \os M) \\ &\leq& (1+\varepsilon) \s \alpha /(z;\os E).\end{eqnarray*}
Since this holds for arbitrary $\varepsilon$, we have $\s \alpha / = \overrightarrow{\s \alpha /}$ on $\os E$.

$(4)$ Is a direct consequence of Lemma~\ref{lemma con palitos}.

$(5)$ Let us see first that $(\s \alpha /)'$ is injective. Consider an isometric embedding $E \overset 1 \hookrightarrow F$ and $z \in \os M$, where $M$ is a finite dimensional subspace of $E$.
Fix $\varepsilon >0$, since $(\s \alpha /)'$ is finitely generated we can take $N \in FIN(F)$ such that $z \in \os N$ and
$$ (\s \alpha /)'(z;\os N) \leq (\s \alpha /)'(z;\os F) + \varepsilon.$$

Denote by $S$ the finite dimensional subspace of $F$ given by $M + N$ and $i: M \to S$ the canonical inclusion.
Observe that $\os i': \os_{\s \alpha /} S' \overset 1 \twoheadrightarrow \os_{\s \alpha /} M'$ is a quotient mapping since the s-tensor norm $\s \alpha /$ is projective. Therefore, its adjoint
$$ (\os i')' : \big(\os_{\s \alpha /}M'\big)'  \overset 1 \hookrightarrow  \big(\os_{\s \alpha /}S'\big)',$$
is an isometric embedding.
Using the definition of the dual norm on finite dimensional spaces and the right identifications, it is easy to show that the following diagram commutes

\begin{equation}
\xymatrix{
{\os_{(\s \alpha /)'}M}  \ar@{=}[d] \ar[rr]^{\;\;\;\;\;\;\;\;\;\;\;  \os i \;\;\;\;\;\;\; } & & {\os_{(\s \alpha /)'}S} \ar@{=}[d] \\
{\big(\os_{\s \alpha /}M'\big)'} \ar@{^{(}->}[rr]^{\;\;\;\;\;\;\;\;\;\;\; (\os i')' \;\;\;\;\;\;\;  } & & {\big(\os_{\s \alpha /}S'\big)'}  \\
}.
\end{equation}
Therefore $\os i : \os_{(\s \alpha /)'}M \to \os_{(\s \alpha /)'}S$ is also an isometric embedding. With this $(\s \alpha /)'(z;\os M) = (\s \alpha /)'(z;\os S)$.
Now,
\begin{align*}
(\s \alpha /)'(z;\os E) & \leq (\s \alpha /)'(z;\os M) \leq (\s \alpha /)'(z;\os S) \\
& \leq (\s \alpha /)'(z;\os N) \leq  (\s \alpha /)'(z;\os F) + \varepsilon.
\end{align*}
Since this holds for every $\varepsilon >0$, we obtain $(\s \alpha /)'(z;\os E) \leq (\s \alpha /)'(z;\os F)$. The other inequality always holds, so $(\s \alpha /)'$ is injective.

We now show that $(\s \alpha /)'$ coincides with $ / \alpha' \s$.
Note that for $m \in \mathbb{N}$,
$$ \os_{(\s \alpha /)'} \ell_{\infty}^m = \big( \os_{\s \alpha /} \ell_{1}^m \big)' = \big( \os_{ \alpha } \ell_{1}^m \big)' = \os_{\alpha '} \ell_{\infty}^m =  \os_{/ \alpha ' \s} \ell_{\infty}^m .$$
Therefore, the s-tensor norms $(\s \alpha /)'$ and $/ \alpha' \s$ coincide in $\os \ell_{\infty}^m$ for every $m \in \mathbb{N}$ and, by Lemma~\ref{lemma con palitos}, their corresponding injective associates coincide. But both  $ (\s \alpha /)' $ and   $/ \alpha' \s$  are injective, which means that they actually are their own injective associates, therefore $ (\s \alpha /)' $ and   $/ \alpha' \s$  are equal.

\bigskip
Let us finally prove that $(/ \alpha \s)' = \s \alpha' /$.
We already showed that $(\s \beta /)' = / \beta ' \s$ for every tensor norm $\beta$. Thus, for $\beta = \alpha'$ we have
$(\s \alpha' /)' = / \alpha '' \s = / \overrightarrow{\alpha} \s = / \alpha \s,$
where the third equality comes from $(1)$.
Thus, by duality, the fact that $\s \alpha' /$ is finitely generated (by $(2)$) and equation~(\ref{bidual de una norma tensorial}) we have
$$ \s \alpha' /= \overrightarrow{\s \alpha' /} = (\s \alpha' /)'' = \big((\s \alpha' /)'\big)' = (/ \alpha \s)',$$
which is what we wanted to prove.
\end{proof}

As a  consequence of Proposition~\ref{relaciones felchita con palito} we obtain the following.

\begin{corollary} \label{observacionespiolas}
Let $\alpha$ be an s-tensor norm. The following holds:
\begin{enumerate}
\item If $\alpha$ is injective then it is finitely and cofinitely generated.
\item If $\alpha$ then is projective then it is finitely generated.
\item If $\alpha$ is finitely or cofinitely generated then: $\alpha$ is injective if and only if $\alpha'$ is projective.
\end{enumerate}
\end{corollary}
\begin{proof}
Note that $(1)$ is a consequence of $(1)$ and $(2)$ of Proposition~\ref{relaciones felchita con palito} and that $(2)$ follows from $(3)$ of Proposition~\ref{relaciones felchita con palito}.
Let us show $(3)$.
If $\alpha$ is injective, we have $\alpha = / \alpha \s$. Thus, we can use  $(5)$ of Proposition~\ref{relaciones felchita con palito} to take dual norms:
$$\alpha' = ( / \alpha \s ) ' = \s \alpha' /.$$
Since the last s-tensor norm is projective, so is $\alpha'$.
Note that for this implication we have not used the fact that $\alpha$ is finitely or cofinitely generated.

Suppose now that $\alpha$ is finitely generated and $\alpha'$ is projective (i.e. $\alpha' = \s \alpha'/$). Thus, by $(5)$ in Proposition~\ref{relaciones felchita con palito} we have
$$\alpha'' = ( \s \alpha' / )' = / \alpha '' \s. $$
Since $\alpha$ is finitely generated, we have $\alpha = \alpha''$, see equation~(\ref{bidual de una norma tensorial}). Thus, $\alpha = / \alpha  \s$, which asserts that $\alpha$ is injective.

Finally, suppose that $\alpha$ is cofinitely generated and $\alpha'$ is projective. Consider an isometric embedding $i: E \overset 1 \hookrightarrow F$.
Since $\alpha'$ is projective,  $\os i' : \os F' \overset 1 \twoheadrightarrow \os E'$ is a quotient mapping and, therefore,  its adjoint $(\os i')'$ is an isometry.
Consider the commutative diagram
\begin{equation}
\xymatrix{
{\os_{\alpha}E = \os_{\overleftarrow{\alpha}}E}  \ar@{-->}[d]^{\os i} \ar@{^{(}->}[rr]^{;\;\;\;\;\;\;\;\;\;\;  1 \;\;\;\;\;\;\;} & & {\big(\os_{\alpha'}E'\big)' } \ar[d]^{(\os i')'} \\
{\os_{\alpha}F=\os_{\overleftarrow{\alpha}}F}  \ar@{^{(}->}[rr]^{ ;\;\;\;\;\;\;\;\;\;\;  1 \;\;\;\;\;\;\; } & & {\big(\os_{\alpha'}F' \big)'} \\
}.
\end{equation}
By the Duality Theorem~\ref{duality theorem} the horizontal arrows are also isometries. This forces $\os i$ to be also an isometry, which means that $\alpha$ respects subspaces isometrically. In other words, $\alpha$ is injective.
\end{proof}

\section{Applications to the theory of polynomial ideals}\label{seccion-polinomios}

In this section we compile some consequences of the previous results to the theory of polynomial ideals.
Let us recall some definitions from \cite{Flo02(On-ideals)}: a \emph{Banach ideal of continuous scalar valued
$n$-homogeneous polynomials} is a pair
$(\mathcal{Q},\|\cdot\|_{\mathcal Q})$ such that:
\begin{enumerate}
\item[(i)] $\mathcal{Q}(E)=\mathcal Q \cap \mathcal{
P}^n(E)$ is a linear subspace of $\mathcal{P}^n(E)$ and $\|\cdot\|_{\mathcal Q}$ is a norm which makes the pair
$(\mathcal{Q},\|\cdot\|_{\mathcal Q})$ a Banach space.

\item[(ii)] If $T\in \mathcal{L} (E_1,E)$, $p \in \mathcal{Q}(E)$ then $p\circ T\in \mathcal{Q}(E_1)$ and $$ \|
p\circ T\|_{\mathcal{Q}(E_1)}\le  \|p\|_{\mathcal{Q}(E)} \| T\|^n.$$

\item[(iii)] $z\mapsto z^n$ belongs to $\mathcal{Q}(\mathbb K)$
and has norm 1.
\end{enumerate}

Let $(\mathcal{Q},\|\cdot\|_{\mathcal Q})$ be the Banach ideal of continuous scalar valued
$n$-homogeneous polynomials and, for $p \in \mathcal{
P}^n(E)$, define
$\|p\|_{\Q^{max}(E)}:= \sup \{ \|p|_M\|_{\Q(M)} : M \in FIN(E) \} \in [0, \infty].$
The maximal hull of $\mathcal{Q}$ is the ideal given by $ \mathcal{Q}^{max} := \{p \in \P : \|p\|_{\Q^{max}} < \infty \}$.
An ideal $\mathcal{Q}$ is said to be \emph{maximal} if $\mathcal{Q} \overset{1}{=} \mathcal{Q}^{max}$.
It is immediate that for each space $E$ , we have $\Q(E) \subset \Q^{max}(E)$. Moreover, $$\| p \|_{\Q^{max}(E)} \leq \|p\|_{\Q(E)}\; \text{for every} \; p \in \Q(E).$$

\emph{The minimal kernel of $\mathcal{Q}$} is defined as the composition ideal $\mathcal{Q}^{min} := \mathcal{Q} \circ \overline{\mathfrak{F}}$, where $\overline{\mathfrak{F}}$ stands for the ideal of approximable operators. In other words, a polynomial $p$ belongs to $\mathcal{Q}^{min}(E)$ if it admits a factorization
\begin{equation} \label{factominimal}
\xymatrix{ E  \ar[rr]^{p} \ar[rd]^{T} & & {\mathbb{K}} \\
& {F} \ar[ru]^{q} & },
\end{equation}
where $F$ is a Banach space, $T : E \to F$ is an
approximable operator and $q$ is in $\Q(F)$.
The minimal norm of $p$ is given by $\|p\|_{\Q^{min}} := \inf \{  \|q\|_{\Q(F)} \|T\|^n \}$, where the infimum runs over all possible factorizations as in~(\ref{factominimal}). An ideal $\mathcal{Q}$ is said to be minimal if $\mathcal{Q} \overset{1}{=} \mathcal{Q}^{min}$.
For more properties about maximal and minimal ideals of homogeneous polynomials and examples see  \cite{Flo02(Ultrastability),Flo01} and the references therein.

If  $\Q$ is a Banach polynomial ideal, its \emph{associated s-tensor norm} is the unique finitely generated tensor norm $\alpha$ satisfying $$\Q(M) \overset 1 =  \otimes^{n,s}_{\alpha} M,$$ for every finite dimensional space $M$.
Notice that $\Q$, $\Q^{max}$ and $\Q^{min}$ have the same associated s-tensor norm since they coincide isometrically on finite dimensional spaces.
The polynomial representation theorem \cite[3.2]{Flo02(Ultrastability)} asserts that, if $\Q$ is maximal, then we have \begin{equation}\label{representation theorem}\Q(E) \overset 1 =  \big( \widetilde{\otimes}^{n,s}_{\alpha'} E \big)',\end{equation} for every space $E$.

A natural question is whether a polynomial ideal is closed under the Aron-Berner extension and, also,  if the ideal norm is preserved by this extension. Positive answers for both questions were obtained for particular polynomial ideals in \cite{CarandoZalduendo99,Carando99,Moraes84} among others. However, some polynomial ideals are not closed under Aron-Berner extension (for example, the ideal of weakly sequentially continuous polynomials). Since the dual s-tensor norm $\alpha'$ is finitely generated, we can rephrase the Extension Lemma~\ref{extension lemma} in terms of maximal polynomial ideals and give a positive answer to the question for ideals of this kind.

\begin{theorem}\textbf{\emph{ (Extension lemma for maximal polynomial ideals.)}}\label{ABisom} Let $\Q$ be a maximal ideal of
$n$-homogeneous polynomials and $p \in \Q(E)$, then its
Aron-Berner extension is in $\Q(E'')$ and
$$\|p\|_{\Q(E)}=\|AB(p)\|_{\Q(E'')}.$$
\end{theorem}

Floret and Hunfeld showed in \cite{Flo02(Ultrastability)} that there is another extension to the bidual, the so called uniterated Aron-Berner extension, which is an isometry for maximal polynomial ideals.
The isometry and other properties of the uniterated extension are rather easy to prove. However, this extension is hard to compute, since its definition depends on an ultrafilter. On the other hand, the Aron-Berner extension is not only easier to compute, but also has a simple characterization that allows to check if a given extension of a polynomial is actually its Aron-Berner extension~\cite{Zalduedo90}. Moreover, the iterated nature of the Aron-Berner extension makes it more appropriate for the the study of polynomials and analytic functions. The next result shows that the Aron-Berner extension is also an isometry for minimal polynomial ideals.

\begin{theorem}\textbf{\emph{ (Extension lemma for minimal polynomial ideals.)}}\label{isom AB minimales}
Let $\Q$ be a minimal ideal. For $p \in {\Q}(E)$, its Aron-Berner extension
$AB(p)$ belongs to ${\Q}(E'')$ and
$$\|p\|_{{\Q}(E)}=\|AB(p)\|_{{\Q}(E'')}.$$
\end{theorem}

\begin{proof}
Since $p \in {\Q}(E)\overset{1}
=\big(({\Q}^{max})^{min}\big)(E)$ (see
\cite[3.4]{Flo01}), given $\varepsilon > 0$ there exist a Banach space $F$, an
approximable operator $T : E \to F$ and a polynomial $q \in
{\Q}^{max}(F)$ such that $p = q\circ T$ (as in (\ref{factominimal})) such that
$ \|q\|_{{\Q}^{max}(F)} \|T\|^n \leq
\|p\|_{{\Q}(E)} + \varepsilon$.

Notice that
$AB(p) = AB(q) \circ T''$. By Theorem~\ref{ABisom} we have
$\|q\|_{{\Q}^{max}(F)}=\|AB(q)\|_{{\Q}^{max}(F)}$.
Since $T$ is approximable, so is $T''$. With this we conclude, that $AB(p)$ belongs to $\Q(E'')$ and
\begin{align*}
\|AB(p)\|_{{\Q}(E'')} & \leq   \|AB(q)\|_{{\Q}^{max}(F'')} \|T''\|^n & \\
& =  \|q\|_{{\Q}^{max}(F)} \|T\|^n \\
& \leq \|p\|_{{\Q}(E)} +
\varepsilon,
\end{align*}
for every $\varepsilon$. The reverse inequality is immediate.
\end{proof}

The next statement is a polynomial version of the Density Lemma~\ref{density}.

\begin{lemma}\textbf{\emph{(Density lemma for
maximal polynomial ideals.)}}
Let $\Q$ be a polynomial ideal, $E$ a Banach space, $E_0$ a dense subspace and $C \subset FIN(E_0)$ a cofinal subset.
Then
$$\|p\|_{\Q^{max}(E)} = \sup \{\|p|_{M}\|_{\Q(M)} : M \in C \}.$$
\end{lemma}

\begin{proof}
For $\alpha$ the s-tensor norm associated to $\Q$, by the Representation Theorem \cite[3.2]{Flo02(Ultrastability)} we have as in (\ref{representation theorem}):
$$  \Q^{max}(E)=(\os_{\alpha'} E)'.$$
Using the Density Lemma \ref{density} (since $\alpha'$ is finitely generated) we get
$$\|p\|_{\Q^{max}(E)} = \|p\|_{(\os_{\alpha'} E)'}=  \|p\|_{(\os_{\alpha'} E_0)'}=\|p\|_{\Q^{max}(E_0)}.$$
On the other hand, by the very definition of the norm in $\Q^{max}$, we have $$\|p\|_{\Q^{max}(E_0)} = \sup \{\|p|_{M}\|_{\Q(M)} : M \in C \},$$ which ends the proof.
\end{proof}

From the previous Lemma we obtain the next  useful result: in the case of a Banach space with a Schauder basis, a polynomial belongs to a maximal ideal if and only if the norms of the the restrictions of the polynomial to the subspaces generated by the first elements of the basis, are uniformly bounded.
\begin{corollary} \label{corolario density lemma}
Let $\Q$ a maximal polynomial ideal, $E$ a Banach space with Schauder basis $(e_k)_{k=1}^\infty$ and $M_m$ the finite dimensional subspace generated by the first $m$ elements of the basis, i.e. $M_m: = [e_k : 1 \leq k \leq m]$.
A polynomial $p$ belongs to $\Q(E)$ if and only if $\sup_{m\in \mathbb{N}} \|p|_{M_m}\|_{\Q(M_m)} < \infty$. Moreover,
$$\|p\|_{\Q(E)}=\sup_{m\in \mathbb{N}} \|p|_{M_m}\|_{\Q(M_m)}.$$
\end{corollary}

\bigskip

As a consequence of the Duality Theorem \ref{duality theorem} we have the following.

\begin{theorem} \label{embedding theorem}
\textbf{\emph{(Embedding Theorem.)}}
Let $\Q$ be the maximal polynomial ideal associated to the s-tensor norm $\alpha$. Then the relations
$$ \os_{\overleftarrow{\alpha}}E  \hookrightarrow \Q(E')$$
$$ \os_{\overleftarrow{\alpha}}E' \hookrightarrow \Q(E),$$
hold isometrically.

In particular, the extension
$$\osc_{\alpha}E' \rightarrow \Q(E)$$ of $\os_{\alpha}E' \rightarrow \Q(E)$ is well defined and has norm one.
\end{theorem}

The following result shows how dominations between s-tensor norms translate into inclusions between maximal polynomial ideals, and viceversa.

\begin{proposition} \label{relacion ideales y normas}
Let $\Q_1$ and $\Q_2$ be maximal polynomial ideals with associated tensor norms $\alpha_1$ and $\alpha_2$ respectively, $E$ be a normed space and $c \geq 0$. Consider the following conditions
\begin{enumerate}
\item $\alpha_2' \leq c \alpha_1'$ on $\os E$;
\item $\Q_2(E) \subset  \Q_1(E) $ and $\| \; \|_{\Q_1} \leq c \; \| \; \|_{\Q_2};$
\item $ \overleftarrow{\alpha_1} \leq c \; \overleftarrow{\alpha_2}$ on $\os E'$.
\end{enumerate}
Then
\begin{enumerate}
\item[(a)] $(1) \Leftrightarrow (2) \Rightarrow (3)$.
\item[(b)] If $E'$ has the metric approximation property then $(1), (2)$ and $(3)$ are equivalent.
\end{enumerate}
\end{proposition}

\begin{proof}
$(a)$ The statement $(1) \Leftrightarrow (2)$ can be easily deduced from the Representation Theorem (see (\ref{representation theorem}) above).

Let us show $(2) \Rightarrow (3)$. Let $z \in \os E'$. By the Embedding Theorem \ref{embedding theorem} we have:
$$ \os_{\overleftarrow{\alpha_1}}E' \overset{1}{\hookrightarrow} \Q_1(E),$$
$$ \os_{\overleftarrow{\alpha_2}}E' \overset{1}{\hookrightarrow} \Q_2(E).$$
Denote by $p_z \in \mathcal{P}(^nE)$ the polynomial that represents $z$.
Thus, $$\overleftarrow{\alpha_1}(z) = \|p_z\|_{\Q_1(E)} \leq c \|p_z\|_{\Q_2(E)} = c \overleftarrow{\alpha_2}(z).$$

$(b)$ Since $E'$ has the metric approximation property, so does $E$ \cite[Corollary 1 in 16.3.]{DefFlo93}. Thus, by Proposition \ref{coinciden con m.a.p.}, for $i=1,2$ we have  $\overrightarrow{\alpha_i}=\alpha_i$ and $\overleftarrow{\alpha_i'}=\alpha_i'$ on $\os E'$ and $\os E$ respectively.  Condition $(3)$ states that the mapping $(**)$ in the following diagram has norm at most $c$.
\begin{equation}\label{diamaccess}
\xymatrix{
{\os_{\alpha_1'}E=\os_{\overleftarrow{\alpha_1'}}E}  \ar@{-->}[d]^{(*)} \ar@{^{(}->}[rr]^{ 1 \;\;\;\;\;\;\;\;\;\; } & & {(\os_{\alpha_1''}E')' = (\os_{\overrightarrow{\alpha_1}}E')'= (\os_{\alpha_1}E')'} \ar[d]^{(**)} \\
{\os_{\alpha_2'}E=\os_{\overleftarrow{\alpha_2'}}E}  \ar@{^{(}->}[rr]^{ 1 \;\;\;\;\;\;\;\;\;\;} & & {(\os_{\alpha_2''}E')' = (\os_{\overrightarrow{\alpha_2}}E')'= (\os_{\alpha_2}E')'} \\
},
\end{equation}
Since the diagram commutes we can conclude that the mapping $(*)$ is continuous with norm $\leq c$. Therefore $(3)$ implies $(1)$.
\end{proof}

The previous proposition is a main tool for translating results on s-tensor norms into results on polynomial ideals. As an example, we have the following polynomial version of the  $\mathcal{L}_p$-Local Technique Lemma~\ref{local technique}.

\begin{theorem} \label{local technique for ideals}
\textbf{\emph{($\mathcal{L}_p$-Local Technique Lemma for Maximal ideals.)}}
Let $\Q_1$ and $\Q_2$ be polynomial ideals with $\Q_1$ maximal  and let $c>0$. Consider the following assertions.

(a) $\| \; \|_{\Q_1(\ell^m_p)} \leq c \; \| \; \|_{\Q_2(\ell^m_p)}$ for all $m\in\mathbb N$.

(b) $\Q_2(\ell_p) \subset \Q_1(\ell_p)$ and $\| \; \|_{\Q_1(\ell_p)} \leq c \; \| \; \|_{\Q_2(\ell_p)}$.

Then (a) and (b) are equivalent and imply that $$\Q_2(E) \subset \Q_1(E)\text{ and } \| \; \|_{\Q_1(E)} \leq c\lambda^n  \; \| \; \|_{\Q_2(E)}$$ for every $\mathcal{L}_{p,\lambda}^g$ normed space $E$.
\end{theorem}
\begin{proof}
Using Corollary~\ref{corolario density lemma} we easily obtain that (a) implies (b).

On the other hand, since the subspace spanned by the first $m$ canonical vectors in $\ell_p$ is a 1-complemented subspace isometrically isomorphic to $\ell_p^m$, we get that (b) implies (a) by the metric mapping property.

Let us show that $(a)$ implies the general conclusion. Denote by $\alpha_1$ and $\alpha_2$ the s-tensor norms associated to $\Q_1$ and $\Q_2$ respectively. By (a) and the Representation Theorem for maximal polynomial ideals \eqref{representation theorem}, we have $\alpha_1' \leq c \: \alpha_2'$ on $\os \ell_{p}^m$.
Using the $\mathcal{L}_p$-Local Technique Lemma~\ref{local technique} we get $\alpha_2' \leq c \lambda^n \alpha_1'$ on $\os E$.
Notice that $\alpha_2$ is also associated with $(\Q_2)^{max}$, thus by Proposition~\ref{relacion ideales y normas} we obtain $(\Q_2)^{max}(E) \subset \Q_1(E)\text{ and } \| \; \|_{\Q_1(E)} \leq c\lambda^n  \; \| \; \|_{(\Q_2)^{max}(E)}.$
Since $\Q_2(E) \subset (\Q_2)^{max}(E)$ and $\| \; \|_{(\Q_2)^{max}(E)} \leq \| \; \|_{\Q_2(E)}$, we finally obtain
$\Q_2(E) \subset \Q_1(E)$ with $\| ; \|_{\Q_1(E)} \leq c\lambda^n  \; \| \; \|_{\Q_2(E)}.$
\end{proof}

\smallskip

For the case $p=\infty$, $\ell_p$ in assertion (b) should be replaced by $c_0$. Moreover,  $\ell_\infty$ is a $\mathcal{L}_{\infty,1}^g$-space and $\ell_\infty^n$ is 1-complemented in $\ell_\infty$ for each $n$. So, in particular, we have: two maximal ideals coincide on $c_0$ if and only if they coincide on $\ell_\infty$.
Let us introduce a polynomial ideal for which the last remark applies.

A polynomial $p \in \mathcal{P}^n$ is \emph{extendible} \cite{KirRy98}
if for any Banach space $F$ containing $E$ there exists
$\widetilde p \in {\mathcal P}^n(F)$ an extension of $p$. We will
denote the space of all such polynomials by ${\mathcal
P}_e^n(E)$. For $p\in {\mathcal P}_e^n(E)$, its extendible
norm is given by
$$\begin{array}{rl}
\Vert p\Vert _{{\mathcal P}_e^n(E)}=\inf \{C>0:  & \mbox{for
all }F\supset E
\mbox{ there is an extension of }p\mbox{ to }F  \\
& \mbox{ with norm}\leq C\}.
\end{array}$$

With this definition, every polynomial on $\ell_\infty$ is extendible, since $\ell_\infty$ is an injective Banach space. Therefore, although $c_0$ is not injective, we get that every polynomial on $c_0$ is extendible (by our previous comment). We remark that  the extendibility of polynomials on $c_0$ is a known fact, and that it can also be obtained from the Extension Lemma.

Since Hilbert spaces are $\mathcal{L}_{p}^g$ for any $1<p<\infty$ (see Corollary 2 in \cite[23.2]{DefFlo93}), we obtain also the following.

\begin{corollary}\label{ele2}
Let $\Q_1$ and $\Q_2$ be polynomial ideals, $\Q_1$ maximal. If for some  $1 < p < \infty$ we have
$\Q_2(\ell_p) \subset \Q_1(\ell_p)$, then we also have $\Q_2(\ell_2) \subset \Q_1(\ell_2)$.
\end{corollary}
As a consequence, if two maximal polynomial ideals do not coincide on $\ell_2$, then they are different in every $\ell_p$ with $1<p<\infty$.

\medskip

Proposition~\ref{relacion ideales y normas},  Theorem~\ref{local technique for ideals} and Corollary~\ref{ele2} have their analogues for minimal ideals. For Theorem~\ref{local technique for ideals} and Corollary~\ref{ele2}, the hypothesis on maximality of  $\Q_1$  should be changed for the requirement that $\Q_2$ be minimal.

\bigskip

We end this note with a few words about accessibility of s-tensor norms and polynomial ideals.
We will say that an s-tensor norm $\alpha$ is \emph{accessible} if $\overrightarrow{\alpha} = \alpha = \overleftarrow{\alpha}$.
For example, item (1) of Corollary~\ref{observacionespiolas} implies that every injective s-tensor norm is accessible.

The definition of accessible polynomial ideals is less direct.
By $\P_f$ we will denote the class of finite type polynomials.
We say that a polynomial ideal $\Q$ is \textit{accessible} (a term coined in \cite[3.6.]{Flo01}) if the following condition holds:
for every normed space $E$, $q \in \p_f(E)$ and $\varepsilon > 0$, there is a closed finite codimensional space $L \subset E$ and $p\in \P(E/L)$ such that $q = p \circ Q_{L}^E$ (where $Q_{L}^E$ is the canonical quotient map) and $\|p\|_{\q}~\leq~(1 + \varepsilon)~\|q\|_{\q}.$

In \cite[Proposition 3.6.]{Flo01} it is shown that, if $\Q$ is accesible then
$$ \Q^{min}(E) \overset 1 \hookrightarrow \Q(E) \; \; \text{and} \; \; \Q^{min}(E) = \overline{\P_f(E)}^{\Q},$$
for every Banach space $E$.
Or, in other words,
$$ \osc_{\alpha} E' =\Q^{min}(E) \overset 1 \hookrightarrow \Q(E),$$
where $\alpha$ is the s-tensor associated to $\Q$. This `looks like' the embedding Theorem \ref{embedding theorem}.

One may wonder how the definition of accessibility of a polynomial ideal relates with the one for its associated s-tensor norm.
The next proposition sheds some light on this question.

\begin{proposition}
Let  $\Q$ be a  polynomial ideal and let $\alpha$ be its associated s-tensor norm. Then, $\alpha$ is accessible if and only if $\Q^{max}$ is, in which case $\Q$ is also accessible.
\end{proposition}

\begin{proof}
Suppose that $\alpha$ is accessible. Then $\alpha$ is finitely and cofinitely generated. Fix $E$ a normed space, $q \in \p_f(E)$ and $\varepsilon > 0$.
Let $z \in \os_{\alpha} E'$ the representing tensor of the polynomial $q$.
Since $\alpha$ is cofinitely generated, by the Duality Theorem \ref{duality theorem} and the Representation Theorem (\ref{representation theorem}) we have
$$ \os_{\alpha} E' \overset{1}{\hookrightarrow} (\os_{\alpha'} E)' = \Q^{max}(E).$$
Thus, $\alpha(z; \os E') = \|q\|_{\Q^{max}(E)}.$
Using that $\alpha$ is finitely generated we can find $M \in FIN(E')$ such that $z \in \os M$ and
$$\alpha(z; \os M) \leq (1+\varepsilon) \|q\|_{\Q^{max}(E)}.$$
Let $L:=M^{0} \subset E$, identifying $M'$ with $E/L$ and denote $p$ the representing polynomial of the tensor $z \in \os M$ defined in $E/L$. Therefore
$\|p\|_{\Q^{max}(E/L)} = \alpha(z; \os M) \leq (1+\varepsilon) \|q\|_{\Q^{max}(E)}$ and obviously $q = p \circ Q_{L}^E$ where $Q_E^L: E\to  E/L$ is the natural quotient mapping.

For the converse we must show that
$\alpha(\; \cdot \;; \os E) = \overleftarrow{\alpha}(\; \cdot \;, \os E).$
By the Embedding Lemma~\ref{embedding} it is sufficient to prove that $\alpha(\; \cdot \;, \os E'') = \overleftarrow{\alpha}(\; \cdot \;, \os E'').$
We will denote $E'$ by $F$, let $z \in \os F'$ and $\varepsilon > 0$.
By the Duality Theorem~\ref{duality theorem}
$$ \os_{\overleftarrow{\alpha}}F' \overset{1}{\hookrightarrow} (\os_{\alpha '} F)'=\Q^{max}(F).$$
Denote by $q$ the polynomial that represents $z$ in $\Q^{max}(F)$; by hypothesis there exist a subspace $L \in COFIN(F)$ and a polynomial $p \in \Q^{max}(F/L)$ such that $q=p \circ Q_{L}^F$ with $\|p\|_{\Q^{max}(F/L)}\leq (1 + \varepsilon) \| q \|_{\Q^{max}(F)}$. If $w$ is the tensor that represents $p$ in $\os L^0 = \os (F/L)'$, we have $(\os Q_{L}^F)(w) = z$. By the metric mapping property,
\begin{align*}
\alpha(z; \os F) & \leq \alpha (w ; \os F/L) \\
& = \|p\|_{\Q^{max}(F/L)}\\
& \leq (1 + \varepsilon) \| q \|_{\Q^{max}(F)}\\
& = (1 + \varepsilon) \overleftarrow{\alpha}(z; \os F),
\end{align*} which proves that $\alpha$ is accessible.

Finally, we always have $\|\cdot\|_{\Q^{max}} \le \|\cdot\|_{\Q}$, with equality in finite dimensional spaces. The definition of accessibility then implies that, if $\Q^{max}$ is accessible, then so is $\Q$.
\end{proof}

\end{document}